\pdfoutput=1

\documentclass[11pt, a4paper]{amsart}
\usepackage[T1]{fontenc}
\usepackage[utf8]{inputenc} 
\usepackage{amssymb, amsmath, amsthm, bm, hyperref, graphicx, color, wrapfig, xparse}
\usepackage{mathtools} 
\usepackage[all, cmtip]{xy}
\usepackage{microtype}

\usepackage[hyperref, doi=false, backrefstyle=two, style=alphabetic, backend=bibtex, eprint]{biblatex}
\newtheorem{Th}{Theorem}[section]
\newtheorem*{Th*}{Theorem}
\newtheorem*{MainTh}{Main Theorem}

\newtheorem{Lem}[Th]{Lemma}
\newtheorem{Prop}[Th]{Proposition}
\newtheorem{Cor}[Th]{Corollary}
\newtheorem*{Cor*}{Corollary}

\theoremstyle{definition}

\newtheorem{Def}[Th]{Definition}

\newtheorem*{Question*}{Question}

\theoremstyle{remark}
\newtheorem{Rm}[Th]{Remark}
\newtheorem{Ex}[Th]{Example}

\newcommand{\Z}{\mathbb{Z}}

\newcommand{\Q}{\mathbb{Q}}
\newcommand{\C}{\mathbb{C}}
\newcommand{\R}{\mathbb{R}}
\newcommand{\F}{\mathbb{F}}

\newcommand{\Category}[1]{\textsf{\textbf{#1}}} 
\DeclareMathOperator{\Nat}{Nat} 

\DeclareMathOperator{\Homeo}{Homeo}
\DeclareMathOperator{\Map}{Map}

\DeclareMathOperator{\tot}{tot} 
\DeclareMathOperator{\Tot}{Tot} 
\DeclareMathOperator{\hocolim}{hocolim}
\newcommand{\op}{\mathrm{op}}
\DeclareMathOperator{\cyc}{cyc} 
\newcommand{\Sum}{\Sigma}
\DeclareMathOperator{\rk}{rk}

\newcommand{\ring}{\Bbbk} 
\mathchardef\mhyphen="2D 
\newcommand{\degree}[1]{|#1|} 
\newcommand{\Sing}{\mathrm{Sing}} 
\newcommand{\convergesto}{\Longrightarrow} 
\newcommand{\loops}{\mathcal{L}} 
\newcommand{\SeqOp}{\mathcal{S}} 
\newcommand{\id}{\mathrm{id}} 
\newcommand{\Moore}{\Omega^{Moore}} 

\newcommand{\Top}{\Category{Top}}

\newcommand{\sSet}{\Category{sSet}} 
\newcommand{\Ch}{\Category{Ch}} 

\newcommand{\dga}{\Category{dga}}
\newcommand{\sdga}{\Category{sdga}} 
\newcommand{\rSet}{\Category{rSet}} 
\newcommand{\invsdga}{\Category{i-rdga}} 
\newcommand{\invdga}{\Category{i-dga}} 
\newcommand{\ChHom}{\underline{\Ch}} 

\newcommand{\ford}{\Delta} 
\newcommand{\reflexive}{\Delta R} 
\newcommand{\cyclic}{\Delta C} 
\newcommand{\dihedral}{\Delta D} 
\newcommand{\stdSim}{\delta} 
\newcommand{\stdCyclic}{\delta_C} 
\newcommand{\stdDihedral}{\delta_D}

\newcommand{\htor}{\mathbb{T}or} 

\newcommand{\TMon}{\mathbb{T}} 
\newcommand{\OMon}{\mathbb{O}} 
\newcommand{\OTop}{\OMon\hspace{0.1em}\mhyphen\Top}

\newcommand{\inv}[1]{\overline{#1}} 

\begin{document}

\title{Free loop spaces and dihedral homology}
\author{Massimiliano Ungheretti}
\address{\newline Department of Mathematical Sciences\newline University of Copenhagen\newline
Universitetsparken 5 \newline 2100 Copenhagen, Denmark}
\email{m.ungheretti@math.ku.dk}
\urladdr{\url{ungheretti.com}}


\begin{abstract}We prove an $O(2)$-equivariant version of the Jones isomorphism relating the Borel $O(2)$-equivariant cohomology of the free loop space to the dihedral homology of the cochain algebra. We discuss polynomial forms and a variation of the de Rham isomorphism and use these to do a computation for the 2-sphere.
\end{abstract}

\maketitle

\section{Introduction}
For any space $X$, one may form the (unbased) mapping space $\loops X=\Map(S^1,X)$. These \emph{free loop spaces} have played a big role in geometry, topology and physics; in particular in string theory, string topology, loop groups and the study of geodesics through the use of the energy functional on free loop space. This is exemplified by the celebrated Gromoll--Meyer Theorem \cite{GromollMeyer}, which states that a simply connected closed Riemannian manifold admits infinitely many distinct closed geodesics if the sequence of Betti numbers $\{\rk H^k(\loops X; \ring)\}_{k\geq 0 }$ is unbounded for a field $\ring$. Although many manifolds are covered by this theorem, it remains an open question whether the assumption on the Betti numbers can be dropped.

The Gromoll--Meyer Theorem is proven by studying the infinite dimensional Morse theory of the energy functional $E(\gamma)= \int_{S^1}||\dot{\gamma}(t)||^2 dt$ on $\loops X$ as the critical points of $E$ correspond to closed geodesics. In \cite[p. 350]{BottMorse} Bott proposes that one has to take into account the invariance of the energy functional under rotations and reflections of the circle. And indeed, Rademacher and Hingston have shown that some other classes of Riemannian manifolds also admit infinitely many distinct geodesics by using Borel equivariant homology with respect to the rotations in $\TMon=SO(2)\subset O(2)$. A survey of such results is given in \cite{Oancea}. Although Lusternik and Schnirelmann proved that the 2-sphere with arbitrary Riemannian metric carries at least three distinct closed geodesics, it is not known if there are always infinitely many of them.

Taking into account the full $O(2)$-symmetry could bring a full answer even closer. This is pointed out in Remark 6.4 of \cite{Oancea} in the following way. An example of Katok \cite{Katok} shows that there is a non-symmetric Finsler metric on the $n$-sphere that admits only finitely many distinct closed geodesics. The notion of a non-symmetric Finsler metric is one that generalizes that of a Riemannian metric in a way that breaks the time reversal symmetry, signifying that the full $O(2)$-symmetry is really needed for admitting infinitely many geodesics.

A common tool for computing the (co)homology of $\loops X$ is the homology theory for algebras, Hochschild homology $HH_*$, and its variations like cyclic homology $HC_*$ and negative cyclic homology $HC^-_*$. In particular, we have the following two theorems available.
\begin{Th}[\cite{Goodwillie}, \cite{BF}]\label{Th:Goodwillie}
Let $X$ be a connected space and $\ring$ any ring.
\[H_*(\loops X;\ring) \cong HH_*(S_*(\Moore X;\ring))\]
\[H_*(\loops X_{h\TMon};\ring) \cong HC_*(S_*(\Moore X;\ring))\]
Here $(-)_{h\TMon}$ denotes the Borel construction with respect to the circle group $\TMon=SO(2)$ and $S_*(\Moore X;\ring)$ is the differential graded algebra of singular chains on the associative monoid of Moore loops on $X$ with Pontryagin product.
\end{Th}
\begin{Th}[\cite{Jones}]\label{Th:Jones}
Let $\ring$ be a field and $X$ a simply connected space with finite type homology over $\ring$.
\[H^*(\loops X;\ring) \cong HH_*(S^*(X;\ring))\]
\[H^*(\loops X_{h\TMon};\ring) \cong HC_*^-(S^*(X;\ring))\]
Here $S^*(X;\ring)$ is the differential graded algebra of (normalized) singular cochains with cup product.
\end{Th}

Although the second theorem is somewhat harder to prove, it is often preferable for computational purposes. For instance, the algebra of cochains $S^*(X)$ is smaller than the Moore loops $S_*(\Moore X)$ and rational homotopy theory can be used to give even smaller models for $S^*(X;\Q)$.

As all free loop spaces come with the slightly bigger symmetry group $\OMon=O(2)\subset \Homeo(S^1)$, it is natural to ask what the analogous algebraic descriptions of (co)homology of $\loops X_{h\OMon}$ are. For the case of homology, Dunn gave the following analogue of Theorem \ref{Th:Goodwillie}.
\begin{Th}[\cite{DunnDihedral}]Let $X$ be a connected space and $\ring$ any ring.
\[H_*(\loops X_{h\OMon};\ring) \cong HD_*(S_*(\Moore X;\ring))\]\end{Th}

Here $HD_*$ is a variation of cyclic homology called \emph{dihedral homology} \cite{LodayDihedrale} that allows one to take into account the $\OMon$-action rather than just the $\TMon$-action. Although Hochschild homology and cyclic homology take as input (differential graded) associative algebras, dihedral homology additionally requires the data of an involution on that algebra. In the case above, this data comes from reversing the loops in $\Moore X$.

The aim of this article is to extend Jones' theorem to take into account the $O(2)$-symmetry of $\loops X$.
\begin{MainTh}\label{th:dih_jones}Let $\ring$ be a field and $X$ a simply connected space with finite type homology over $\ring$. Then there is an isomorphism 
\[ H^*(\loops X_{h\OMon};\ring) \cong HD^-_*(S^*(X;\ring)).\]\end{MainTh}
Here $HD_*^-$ denotes a variation of dihedral homology called \emph{negative dihedral homology} and the cochain algebra $S^*(X;\ring)$ carries a homotopically trivial involution coming from changing the orientation of simplices.
\begin{proof}[Outline of the proof] \renewcommand{\qedsymbol}{} \leavevmode
\begin{enumerate}
\item\label{i:loop_model} To model the left hand side we start with the codihedral space $\Map(S^1_\bullet, X)$. A codihedral space is a cosimplicial space with extra structure that allows for an $\OMon$-action on its totalization. This codihedral space is used as a model for free loop space because $\tot \Map(S^1_\bullet, X) \cong_\OMon \loops X$ and hence \[S^*(\loops X_{h\OMon})\cong S^*((\tot \Map(S^1_\bullet, X))_{h\OMon} ).\]
\item\label{i:compare_orbits} The chains on the homotopy orbit space are then compared to an algebraic version of homotopy orbits
\[S^*((\tot \Map(S^1_\bullet, X))_{h\OMon} ) \simeq {S_*(\tot \Map(S^1_\bullet, X))_{h\OMon}}^\vee.\]
\item\label{i:tensor_hom} The tensor-hom adjunction relates the result of the last step to the algebraic homotopy fixed points of the dual
\[{S_*(\tot \Map(S^1_\bullet, X))_{h\OMon}}^\vee \cong S^*(\tot \Map(S^1_\bullet, X))^{h\OMon}.\]
\item\label{i:compare_tot} With the appropriate assumptions, comparing the two ways of totalizing the cochains on $\Map(S^1_\bullet, X)$ yields an equivalence
\[(S^*(\tot \Map(S^1_\bullet, X)))^{h\OMon} \simeq (\Tot_\oplus S^*(\Map(S^1_\bullet, X)))^{h\OMon}.\]
\item\label{i:cyc_bar_free_loops} After proving an equivalence of $S^*(\Map(S^1_\bullet, X))$ with the cyclic bar construction as a dihedral chain complex, it follows that
\[(\Tot_\oplus S^*(\Map(S^1_\bullet, X)))^{h\OMon} \simeq (\Tot_\oplus B^{\cyc}S^*(X))^{h\OMon}.\]
\end{enumerate}
As the homology of the last term is our definition of $HD^-_*(S^*(X))$, the result follows after taking homology.
\end{proof}

Since the cochain algebra $S^*(X;\ring)$ is generally too big to compute with, we prove that one may instead use the polynomial forms $A^*_{PL}(X)$ when $\ring$ is of characteristic 0 as $S^*(X;\ring)\simeq A^*_{PL}(X)$ as involutive algebras. Similarly, we prove that the de Rham isomorphism is compatible with the involutions. The following corollary is particularly useful.

\begin{Cor*}Let $X$ be a rationally formal space. Then there is an isomorphism
\[HD^-_*(S^*(X;\Q))\cong HD^-_*(H^*(X;\Q)).\]
\end{Cor*}

This corollary is used in the last section to compute $H^*((\loops S^2)_{h\TMon};\Q)$, which turns out to be one-dimensional in every dimension $* \equiv 0, 3$ modulo $4$, and zero otherwise. In characteristic two, the answer is more interesting. In that case the dimensions of $HD^-_*(H^*(S^2;\F_2))$ have been computed in low degrees to be the unbounded sequence $\lfloor (*+2)^2/4 \rfloor$. Unfortunately, the author was not able to show that the sphere is involutively formal over $\F_2$ meaning that the calculation does not necessarily apply to $H^*((\loops S^2)_{h\TMon};\F_2)$. With the Gromoll--Meyer Theorem in mind, it does however give another hopeful indicator that the $C_2$-symmetry could help proving that $S^2$ admits infinitely many distinct closed geodesics for any Riemannian metric.

\subsection*{Organization of the paper}
Sections \ref{sec:involutive_algebras}--\ref{sec:dihedral_homology} are dedicated to the definitions of involutions, dihedral objects and (negative) dihedral homology in a way suited to our application and Section \ref{sec:goodwillie} is an outline of the proof of Dunn's result. In Section \ref{sec:cyc_bar_free_loops} we prove Step \ref{i:cyc_bar_free_loops} of the outline, followed by Step \ref{i:compare_orbits} in Section \ref{sec:compare_orbits} and Step \ref{i:compare_tot} in Section \ref{sec:compare_tot}. The full proof of the Main Theorem is then given in Section \ref{sec:proof_main_theorem}. Section \ref{sec:deRham} is dedicated to polynomial forms and a version of the de Rham isomorphism that may aid in computations, which is then used in Section \ref{sec:example_calc} to calculate $H^*((\loops S^2)_{h\OMon};\Q)$.

\subsection*{Conventions}
All algebras are unital over a base ring $\ring$. We use the closed monoidal structure of $\Ch$, the category of unbounded homologically graded chain complexes over $\ring$. For example, the tensor product of two chain complexes has differential $d(x\otimes y)=dx\otimes y + (-1)^{\degree{x}} x \otimes dy$ and $f\in \ChHom(X,Y)_n$ is a map of degree $n$ with differential $(\delta f)(x)=d(f(x))-(-1)^n f(dx)$. The differential on the (sum) totalization of a simplicial chain complex is defined as $d_{\textit{int}}+(-1)^{\textit{int}}\Sigma_i (-1)^i d_i$. Similarly, the product totalization $\Tot_\Pi X^\bullet$ of a cosimplicial chain complex $X^\bullet$ has a differential that is $dx - (-1)^{p-n}\Sum (-1)^i \delta^i x$ on $x\in {X^n}_p$. 

If $Y^\bullet$ is a cosimplicial space, we define $\tot Y^\bullet = \Nat_\ford(\ford^\bullet,Y^\bullet)\in \Top$ to be the totalization.

Let $\TMon$ be the circle group, considered as a subset of the complex numbers. We denote the orthogonal group $O(2)=\TMon\rtimes C_2$ by $\OMon$. In this notation the multiplication on $\OMon$ is $(\tilde{z},\tilde{\alpha})\cdot(z,\alpha)=(\tilde{z}z^{\tilde{\alpha}},\tilde{\alpha}\alpha)$ where we consider $\alpha,\tilde{\alpha}=\pm1\in C_2$.

\subsection*{Acknowledgements}
The author would like to thank Amalie H\o genhaven and Kristian Moi for an invitation to the dihedral world and Nathalie Wahl for general guidance. The author was supported by the Danish National Sciences Research Council (DNSRC) and the European Research Council (ERC), as well as by the Danish National Research Foundation (DNRF) through the Centre for Symmetry and Deformation.

\section{Involutive algebras}\label{sec:involutive_algebras}
\begin{Def}Let $A$ be a differential graded algebra; that is, a monoid in $\Ch$. A chain map $\inv{(-)}\colon A\rightarrow A$ of degree zero is called an \emph{involution} if $\inv{\inv{a}} = a$, $\inv{1}=1$ and $\inv{ab}=(-1)^{\degree{a}\degree{b}}\inv{b}\inv{a}$ for all homogeneous $a,b \in A$. Such a map is called an anti-involution by some due to the flipping of the order. The data of a differential graded algebra together with an involution is called an \emph{involutive algebra}.
\end{Def}
\begin{Ex}\label{ex:cdga}If $A$ is graded commutative, then the identity map is an involution for $A$. In fact, every algebra endomorphism that squares to the identity is an involution.\end{Ex}
\begin{Ex}Complex conjugation is an involution for $\C$ as an algebra over the reals.\end{Ex}
\begin{Ex}\label{ex:singular_cochains_involution}We repeatedly use the differential graded algebra of (normalized) singular cochains $S^*(X)$. Because it is defined as the linear dual of singular chains $S_*(X)=\ring \otimes \Sing_*(X)$, it carries the differential
\[(\delta \gamma)(\sigma)=(-1)^{\degree{\gamma + 1}} \Sum_{i=0}^{n+1} \gamma(d_i \sigma).\]
The cup product is defined as
\[(\gamma_1\cup \gamma_2)(\sigma)= (-1)^{pq}\gamma_1(d_{p+1}\ldots d_{p+q}\sigma)\gamma_2((d_0)^p \sigma),\] 
where $\gamma_1$ and $\gamma_2$ are cochains of degree $p$ and $q$ respectively. This dga carries a natural involution given by $\inv{\gamma}(\sigma)=(-1)^{\degree{\gamma}(\degree{\gamma}+1)/2}\gamma(\inv{\sigma})$ where $\inv{\sigma}$ is the flipped simplex $\inv{\sigma}(t_0,\ldots,t_n)= \sigma(t_n,\ldots,t_0)$. See also Proposition \ref{prop:sing_inv}.\end{Ex}
\begin{Ex}For a group $G$, the map $g\mapsto g^{-1}$ is an involution on the group algebra $\ring G$.\end{Ex}
\begin{Ex}The singular chains of a topological monoid with involution form an involutive dga. An example of this is $S_*(\Moore X; \ring)$.\end{Ex}

\section{Cyclic and dihedral objects}\label{sec:dihedral_objects}
We recall some definitions of cyclic and dihedral objects and refer to \cite{Jones,LodayCyclic,FiedorowiczLodayCrossed,UngherettiCyclicBar} for more details.
The morphisms of the category of finite ordered sets $\ford$ are generated by $\delta^i, \sigma^i$, which satisfy the dual simplicial relations. By appropriately adding cyclic permutations $\langle \tau_n \rangle = C_{n+1}$ as the automorphisms of $[n]$, one obtains Connes' \emph{cyclic category} $\cyclic$. One obtains the \emph{dihedral category} $\dihedral$ if one also adds automorphisms $\rho_n$ at each $[n]$ such that the automorphisms become $D_{n+1}^{\op}$. The morphisms of this category may also be described as maps of unoriented necklaces. The subcategory of $\dihedral$ generated only by the maps $\delta_n^i, \sigma_n^i $ and $\rho_n$ for each $n,i$ is called the \emph{reflexive category}, denoted $\reflexive$. In analogy to the definitions of simplicial and cosimplicial objects, we call a contravariant functor from $\dihedral^{\op}$ to a category $\Category{C}$ a \emph{dihedral object} in $\Category{C}$ and a covariant such functor is called a \emph{codihedral object}. 

The morphisms in $\dihedral$ will be denoted by Greek letters $\delta^i_n, \sigma^i_n, \tau_n, \rho_n$ whereas we use the Roman alphabet for morphisms in the opposite category.
\begin{Ex}
A dihedral set is a simplicial set $X_\bullet$ with extra structure maps $t_n, r_n\colon X_n\rightarrow X_n$ for all $n$ such that the following identities are satisfied:
\begin{equation*}
\begin{aligned}[c]
d_n&=d_0t_n\\
d_it_n&=t_{n-1}d_{i-1}\\
s_it_n&=t_{n+1}s_{i-1}\\
s_0t_n&=t_{n+1}^2 s_n
\end{aligned}
\qquad
\begin{aligned}[c]
t_n^{n+1}&=r_n^2=\id_n\\
rt&=t^{-1}r\\
d_i r_n&= r_{n-1}d_{n-i}\\
s_i r_n&=r_{n+1}s_{n-i}
\end{aligned}
\end{equation*}
\end{Ex}
\begin{Ex} The singular set $\Sing_\bullet X$ of a topological space $X$ is a reflexive set using $r_n(\sigma)(t_0,\ldots,t_n)=\inv{\sigma}(t_0,\ldots,t_n)= \sigma(t_n,\ldots,t_0)$ on an $n\mhyphen$simplex $\sigma$. It is in fact also dihedral, but we do not use this fact. \end{Ex}
\begin{Ex}\label{ex:bar} If $A$ is an algebra with involution, then its bar construction is a reflexive chain complex using $r_n(a_1\otimes \ldots \otimes a_n)=\pm (\inv{a_n}\otimes \ldots \otimes \inv{a_1})$. This construction works for arbitrary monoids with involutions in symmetric monoidal categories.\end{Ex}
\begin{Ex}\label{ex:circle}
The simplicial model for the circle $[n]\mapsto S^1_n = \Z/(n+1)\Z$ is not only a cyclic set, it is also dihedral by using $r_n(i)=n-i+1$, which corresponds to reversing the orientation of the circle. From this dihedral set, one obtains for each space $X$ a codihedral space $[n] \mapsto \Map(S^1_n, X)= X^{n+1}$ that totalizes to the free loop space $\loops X$. The coboundaries are given by the diagonal maps, the codegeneracies by forgetting factors, cyclic maps by cyclically permuting the factors and the reflection by flipping the coordinates. For example, \[\delta^{n+1}(x_0,\ldots,x_n)=(x_0,x_1,\ldots,x_n,x_0).\] By functoriality of $S^*(\mhyphen)$, $S^*(\Map(S^1_\bullet,X))$ is a dihedral chain complex.
\end{Ex}
\begin{Ex}\label{ex:cyclic_bar}For any differential graded algebra $A$, we have the \emph{cyclic bar construction} $(B^{\cyc}A)[n]=A^{\otimes n+1}$, which is used to compute the Hochschild homology of $A$. The structure maps are given by multiplication, insertions of the unit and cyclic permutations of the tensor factors. If the algebra came with an involution, then $B^{\cyc}A$ is also a dihedral chain complex with $r_n(a_0\otimes \ldots \otimes a_n)=(-1)^{\degree{a_n}(\degree{a_1}+\ldots+\degree{a_{n-1}})}\inv{a_0}\otimes \inv{a_n} \otimes \ldots \otimes \inv{a_1}$.\end{Ex}
\begin{Ex}{\cite{FiedorowiczLodayCrossed,Jones}} Composing the Yoneda embedding $\ford\rightarrow \sSet$ with the realization functor $\sSet\rightarrow \Top$ we obtain the \emph{standard cosimplicial space} $\stdSim^\bullet$, which is the geometric standard simplex $\ford^n$ in simplicial degree $n$. Using this object, one can rewrite the geometric realization of a simplicial space $X_\bullet$ as the coend construction $|X_\bullet|= X_\bullet\otimes_{\ford}\stdSim^\bullet$ and the totalization of a cosimplicial space as $\tot Y^\bullet = \Nat_\ford(\stdSim^\bullet,Y^\bullet)$. The same can be done for the dihedral category, obtaining the \emph{standard codihedral space} $\stdDihedral$. Concretely, $\stdDihedral^n\cong \OMon \times \ford^n$ with the following structure maps.
\[
\begin{aligned}[c]
\delta^i&=\id \times \delta^i\\
\sigma^i&=\id \times \sigma^i\\
\rho_n(z,\alpha,t_0,\ldots,t_n)&=(z,-\alpha,t_n,\ldots,t_0)\\
\tau_n(z_\alpha,t_0,\ldots,t_n)&=(z\exp(-\alpha 2\pi i t_0),\alpha,t_1,\ldots,t_n,t_0)
\end{aligned}
\]
It is an elementary check that all of the structure maps are $\OMon$-equivariant if one uses left multiplication. This means that in fact $\stdDihedral$ is a functor $\dihedral\rightarrow \OTop$. The same construction can be done for the cyclic and the reflexive category and the resulting functors are all compatible.
\end{Ex}

\begin{Prop}The realization of a dihedral space has a natural $\OMon$-action. The same is true for the totalization of a codihedral space.\end{Prop}
\begin{proof}See also Theorem 5.3 of \cite{FiedorowiczLodayCrossed} and \S3 of \cite{Jones}. We have that
\[
\begin{aligned}[c]
|X_\bullet|&=X_\bullet \otimes_{\ford} \stdSim^\bullet \cong X_\bullet \otimes_{\dihedral} \stdDihedral^\bullet \\
\tot Y^\bullet &= \Nat_\ford(\stdSim^\bullet,Y^\bullet)\cong \Nat_\dihedral(\stdDihedral^\bullet,Y^\bullet).
\end{aligned}\]
In both cases, the $\OMon$-action is now given by acting on $\stdDihedral^\bullet\cong \OMon \times \ford^\bullet$. Because the action is natural in the structure maps of $\stdDihedral^\bullet$, these actions are well defined and natural.\end{proof}

\begin{Rm}\label{rm:reflexive_action}
Note that every dihedral space is reflexive by forgetting along the inclusion $\reflexive\hookrightarrow\dihedral$. The resulting $C_2$-action is surprisingly simple, given that describing the $\TMon$-action on the realization of a cyclic space is not really explicit in the same way. If $x\in X_n$ and $\underline{t}=(t_0,\ldots,t_n)\in \ford^n$, the action of the generator of $C_2$ on the point $[x,\underline{t}]=[x,+1,\underline{t}]$ is $[x,-1,\underline{t}]=[x,\rho_n(+1,t_n,\ldots,t_0)]=[r_n(x),+1,t_n,\ldots,t_0]$.
\end{Rm}

\begin{Ex}\label{ex:realize_circle}The dihedral set $S^1_\bullet$ from Example \ref{ex:circle} realizes to the circle. The $C_2$-action is the map $z\mapsto z^{-1}$. This can be checked explicitly using the identification $S^1\cong |S^1_\bullet|: z=e^{2\pi i \theta}\mapsto [1, (\theta,1-\theta)]$ where $1$ is the fundamental simplex $1 \in \Z/2\Z = S^1_1$. We also identify the totalization of the codihedral mapping space as $\tot(\Map(S^1_\bullet,X))\cong_{\OMon}\loops X$.\end{Ex}

\section{Cyclic and dihedral homology}\label{sec:dihedral_homology}
Although the use of cyclic homology is widespread, Loday's dihedral homology is less commonly known. The material presented in this section is based on the various treatments in the literature, especially on \cite{KLS,FiedorowiczLodayCrossed,LodderFixedPoints,LodderKTheory,LodayCyclic} and of course the original source \cite{LodayDihedrale}. Although our definitions of dihedral homology and cohomology turn out to coincide with those in the literature, the presentation is somewhat different.

Because our aim is to produce an algebraic model for homotopy orbits of an $\OMon\mhyphen$space, our definitions of cyclic and dihedral homology will be in analogy to constructions in $\Top$. In particular, we will abuse notation by writing $\TMon=H_*(\TMon;\ring)$ and $\OMon=H_*(\OMon;\ring)$ for the graded algebras obtained by applying singular homology to the two topological groups $\TMon$ and $\OMon$. We see that $\TMon$ is generated by $B$, the fundamental class of the circle, which is of degree one and satisfies $B^2=0$. The algebra $\OMon=\TMon\rtimes C_2$ has an additional generator $R$ of degree zero, which satisfies $R^2=0$ and $RB=-BR$. We are especially interested in differential graded modules over these algebras.

\begin{Ex}Let $E\TMon_*$ be the normalized total complex of the two sided bar construction $B_\bullet(\ring,\TMon,\TMon)$. This is a free contractible right differential graded $\TMon\mhyphen$module of the form $\ring[u^{-1}]\otimes \TMon$ where $\degree{u}=-2$ and the differential is $u^{-p} \otimes 1\mapsto u^{-p+1} \otimes 1, u^{-p}\otimes B\mapsto 0$. The group $C_2$ acts on $\TMon$ by $B\mapsto -B$, so the corresponding simplicial action on $E\TMon_*$ is $u^{-p} \otimes 1 \mapsto (-1)^pu^{-p} \otimes 1$ and $u^{-p}\otimes B\mapsto (-1)^{p+1}u^{-p} \otimes B$. That gives us a free contractible right differential graded $\OMon\mhyphen$module $E\OMon_*\coloneqq E\TMon_* \otimes (EC_2)_*$. Here $EC_2$ denotes the periodic resolution of the trivial $C_2$-module that in every non-negative degree is given by $\ring C_2$. The differential on an element of degree $p$ is multiplication by $g+(-1)^p 1$, where $g$ is the generator of $C_2$. We can write $EC_2$ as $\ring C_2 \otimes \ring [v^{-1}]$ with $\degree{v}=-1$ with a non-trivial differential.
\end{Ex}

\begin{Ex} (See also \S4 of \cite{Jones}) Let $W$ be a $\TMon\mhyphen$space with action map $\mu\colon\TMon\times W\rightarrow W$. The formula $B(\sigma)=\mu_*(z\times \sigma)$ defines a left differential graded $\TMon\mhyphen$module structure on the singular chains $S_*(W)$. Here $[\TMon]$ is the fundamental cycle of $\TMon$. If $W$ was an $\OMon\mhyphen$space, the chains also form an $\OMon\mhyphen$module. \end{Ex}

\begin{Ex}\label{ex:total_dihedral}The totalization of a cyclic chain complex is a $\TMon\mhyphen$module and the totalization of a dihedral chain complex is an $\OMon\mhyphen$module. Using the structure maps we may define the following operations in simplicial degree $n$: The simplicial boundary map $b_n=\Sigma_{i=0}^n(-1)^i d_i$, the cyclic generator $ T_n=(-1)^nt_n$, the generator of the $C_2$-action $R_n=(-1)^{n(n+1)/2}r_n$ and the norm operator $N_n=\id + T + T^2 + \ldots T^n$. These in turn allow us to define Connes' B operator $ B_n=(-1)^{\textit{int}}(1-T)t_{n+1}s_nN$. The operations $T$ and $R$ form chain maps with respect to both the internal differential and $b$ whereas $B$ anticommutes with both. The operations also satisfy the relations $(T_n)^{n+1}=R^2=\id, RTR=T^{-1}, B^2=b^2=0, BR=-RB$. The proofs of most of these properties and identities are found in Chapter 2 of \cite{LodayCyclic}.\end{Ex} 

\begin{Def}
Let $M_*$ be a (differential graded) left $\TMon\mhyphen$module. We define $M_{h\TMon} \coloneqq E\TMon_*\otimes_\TMon M_*$ and $M^{h\TMon} \coloneqq \ChHom_\TMon(E\TMon_*,M_*) \subset \ChHom(E\TMon_*,M_*)$. If $M_*$ is moreover an $\OMon\mhyphen$module, we similarly define $M_{h\OMon} \coloneqq E\OMon_*\otimes_\OMon M_*$ and $M^{h\OMon} \coloneqq \ChHom_\OMon(E\OMon_*,M_*)$. 
\end{Def}

\begin{Prop}\label{prop:alg_orbits} Let $M_*$ be a differential graded left $\OMon\mhyphen$module, then
\[M_{h\OMon}=(M_{h\TMon})_{hC_2} \quad \textit{and} \quad M^{h\OMon}=(M^{h\TMon})^{hC_2}.\]
\end{Prop}
\begin{proof}We can convert left into right modules and visa versa using the Hopf algebra structures of $\OMon$ and $\TMon$. Also, we can use the semi direct product structure $\OMon=\TMon \rtimes C_2$ to break down the tensor product over $\OMon$ into two steps $\ring \otimes_\OMon (-) = \ring \otimes_{C_2} (\ring \otimes_\TMon (-))$. We see that
\[M_{h\OMon}=E\OMon \otimes_\OMon M = \ring \otimes_\OMon (E\TMon \otimes EC_2 \otimes M)= \ring \otimes_{C_2} (\ring \otimes_\TMon (E\TMon \otimes EC_2 \otimes M)).\]
Because $\TMon$ acts trivially on the $EC_2$ factor we conclude that \[M_{h\OMon}=EC_2\otimes_{C_2} (E\TMon \otimes_\TMon M)=(M_{h\TMon})_{hC_2}.\] A similar argument shows that $M^{h\OMon}=(M^{h\TMon})^{hC_2}$.
\end{proof}

Combining the proposition with the concrete expressions we see that $M_{h\TMon}=M[u^{-1}]$ with the differential
\[u^{-q}m\mapsto u^{-q}d_Mm + u^{-q+1}Bm,\]
and $M_{h\OMon}=M[v^{-1},u^{-1}]$ with the differential
\begin{align*}
v^{-p}u^{-q}m\mapsto &(-1)^p(v^{-p}u^{-q}d_Mm + v^{-p}u^{-q+1}Bm)\\ 
&+ v^{-p+1}u^{-q}((-1)^qR+(-1)^p)m.
\end{align*}
Note that the degree of $v^{-1}$ is $1$ and the degree of $u^{-1}$ is $2$. The signs in the last term are the differential $g+(-1)^p1$ coming from the periodic resolution of the constant $C_2$-module, applied to $M_{h\TMon}$.

\begin{Def}
Let $A$ be an involutive dga. We define the \emph{Hochschild complex} $C_*(A)$ to be the normalized total complex of the dihedral chain complex $B^{\cyc}A$ of Example \ref{ex:cyclic_bar}. We then define the \emph{cyclic chains} and the \emph{negative cyclic chains} to be $CC_*(A)=C_*(A)_{h\TMon}$ and $CC_*^-(A)=C_*(A)^{h\TMon}$. Similarly we define the \emph{dihedral chains} and \emph{negative dihedral chains} to be $DC_*(A)=C_*(A)_{h\OMon}$ and $DC_*^-(A)=C_*(A)^{h\OMon}$. The corresponding homologies are denoted $HH_*(A), CH_*(A), CH_*^-(A),DH_*(A)$ and $DH_*^-(A)$.
\end{Def}

\begin{Prop}\label{prop:orbit_equivalence}Let $\phi\colon M_*\rightarrow N_*$ be a map of differential graded $\OMon\mhyphen$modules. If $\phi$ is a quasi isomorphism, then the associated maps $\psi_{h\OMon}\colon M_{h\OMon}\rightarrow N_{h\OMon}$ and $\psi^{h\OMon}\colon M^{h\OMon}\rightarrow N^{h\OMon}$ are also quasi isomorphisms.\end{Prop}
\begin{proof}The double complex arguments in the proofs of parts ii and iii of Lemma 2.1 in \cite{Jones} imply that the maps $\psi_{h\TMon}$ and $\psi^{h\TMon}$ are quasi isomorphisms. The same proof can be used to show that the functors $(-)_{hC_2}$ and $(-)^{hC_2}$ preserve quasi isomorphisms. As $\psi_{h\OMon}=(\psi_{h\TMon})_{hC_2}$ and $\psi^{h\OMon}=(\psi^{h\TMon})^{hC_2}$ are compositions of these functors, both maps are quasi isomorphisms.
\end{proof}

\subsection*{Comparison with other definitions}
In the literature, starting with \cite{LodayDihedrale}, it is common to define the dihedral homology of a dihedral $\ring$-module $M_\bullet$ as $HD_*(M_\bullet)=Tor_*^{\dihedral}(\ring^\dagger,M_\bullet)$, where $\ring^\dagger$ denotes the trivial (co)dihedral $\ring$-module. Several different chain complexes are available for computing the homology. In particular, every resolution of $\ring^\dagger$ yields such a chain complex. For example, one could resolve all the dihedral groups and patch them together to get a resolution of the trivial module. For the case when $2$ is invertible in our base ring $\ring$ this is in fact what Loday did in \cite{LodayDihedrale} and a version without this assumption first appeared in \cite{LodderLoops}. When working with cyclic homology it is common to take the cyclic analogue of this complex and contract a subcomplex, obtaining the $(B,b)$-complex. This procedure can also be applied for dihedral homology to obtain a $(B,b)$ version of dihedral chains, see also \cite[Proposition 1.7]{LodayDihedrale} and \cite[Lemma 2.2]{LodderFixedPoints}. In fact, this can be used to see that our definition of dihedral homology is isomorphic to the $Tor$ definition.

In \cite{LodderFixedPoints} Lodder discusses several possible definitions for the negative variant of dihedral homology. One of these is called $\mathcal{D}III$, and our definition of negative dihedral homology coincides with the hyperhomology version of this definition. Although it does not seem to be mentioned explicitly, it seems that $HD_*^-(M)=Ext_{\dihedral}^{-*}(\ring^\dagger, M)$.

\section{Dihedral Goodwillie isomorphism}\label{sec:goodwillie}
This section is a summary of how a dihedral version of the Goodwillie isomorphism is proven in \cite{DunnDihedral}. In \cite{DunnDihedral}, all topological space are assumed to be compactly generated and LEC means that the diagonal map is a cofibration. CW complexes are examples of LEC spaces.
\begin{Th}[\cite{DunnDihedral} Th 3.6] Let $G$ be a group-like topological LEC unital monoid with involution, i.e., with a self map $\inv{(-)}:G\rightarrow G$ satisfying $\inv{a\cdot b}=\inv{b}\cdot \inv{a}$ and $\inv{e}=e$. Then for $\ring$ a ring we have an isomorphism \[HD_*(S_*(G;\ring))\cong H_*((\loops BG)_{h\OMon}; \ring)\]
Here the differential graded algebra $S_*(G;\ring)$ carries the involution induced by the involution of the monoid $G$ and $\loops BG$ has the involution that both reverses the direction of loops and uses $BG\xrightarrow{B \inv{( \cdot) } } BG$.\end{Th}
\begin{proof}The proof can be broken down into a few steps. First we use the Eilenberg--Zilber maps to construct a quasi isomorphism of dihedral chain complexes $B^{cyc}S_*G\xrightarrow{\simeq} S_*(B^{cyc}G)$, which is Proposition 3.5 \cite{DunnDihedral}. Here $B^{cyc}$ denotes the cyclic bar construction, promoted to a dihedral object as in Example \ref{ex:cyclic_bar}. In the first instance this is done in $(\Ch,\otimes)$ and in the second in $(\Top, \times)$. Then we use \cite[3.3]{DunnDihedral}: If $Y_\bullet$ is a (good) dihedral space (e.g., $B^{cyc}G$), then $HD_*(S_*Y_\bullet)\cong H_*(\hocolim_{\ford D}Y_\bullet)$. This can be seen using a statement about hypertor of functors \[HD_*(S_*Y_\bullet)\cong\htor_n^{\ford D}(\ring,S_*Y)\cong H_n(\hocolim_{\ford D} Y_\bullet; \ring),\] which is a generalization of Theorem 6.12 of \cite{FiedorowiczLodayCrossed}. The final ingredient is \cite[2.10]{DunnDihedral} $\hocolim_{\ford D}B^{cyc}G\simeq \loops BG_{h\OMon}$, which follows from the fact that $|B^{cyc}G|\simeq_{\OMon} \loops BG$.\end{proof}

\begin{Prop}There is a natural $C_2$-equivariant map
\[\xi\colon B\Moore X\rightarrow X\]
that is a weak equivalence if $X$ is connected. Here we use the trivial $C_2$-action on $X$ and the action on $B\Moore X$ coming from the involution on $\Moore X$ by reversing loops, see also Example \ref{ex:bar}.\end{Prop}
\begin{proof}We begin by describing the map as defined in \cite[Lemma 15.4]{MayClassifyingSpacesFibrations}. It is also shown there that this map is a weak equivalence if $X$ is connected. By viewing the classifying space as the realization of a bar construction, we may define $\xi [\gamma_1,\ldots ,\gamma_n;u]=(\gamma_1\ldots\gamma_n)(\Sigma_{1\leq i\leq p}u_i a_i)$. Here $\gamma_i\in \Moore X$ of length $a_i$ , $u=(t_0,\ldots,t_p)\in \ford^p$ and $u_i=t_0 + \ldots + t_{i-1}$. The source carries a $C_2$-action because it is the realization of a reflexive object, see also Remark \ref{rm:reflexive_action}. The generator of $C_2$ acts as $[\inv{\gamma_n},\ldots,\inv{\gamma_1};\inv{u}]$ and a quick calculation shows that $\xi([\inv{\gamma_n},\ldots,\inv{\gamma_1};\inv{u}])= \xi([\gamma_1,\ldots,\gamma_p;u])$. From this calculation it is also clear what to do when $X$ is a $C_2$-space: One may simply add this action to the involution of the monoid $\Moore X$.
\end{proof}

\begin{Cor}[Dihedral Goodwillie Isomorphism]For $X$ a connected LEC space and $\ring$ any ring. \[HD_*(S_*(\Moore X);\ring)\cong H_*((\loops X)_{h\OMon};\ring)\]\end{Cor}
\begin{proof}This follows from the theorem above by inserting $G=\Moore X$ and using the equivalence $BG\xrightarrow{\simeq} X$.\end{proof}

\section{The cyclic bar construction and free loop spaces}\label{sec:cyc_bar_free_loops}
In this section we establish an equivalence of dihedral objects between the cyclic bar construction of the cochains (Example \ref{ex:cyclic_bar}) and the cochains of the cosimplicial model for free loop space (Example \ref{ex:circle}). This is done by extending the results of \cite{UngherettiCyclicBar} from an equivalence of (co)cyclic chain complex to an equivalence of (co)dihedral chain complexes.
\begin{Prop}\label{prop:bar_equivalence}
Let $X$ be a space with finite type homology over a principal ideal domain $\ring$. There is a natural zigzag of equivalences of \emph{dihedral} chain complexes \[B^{\cyc}_\bullet S^*(X;\ring) \xleftarrow{\simeq} QB^{\cyc}_\bullet S^*(X;\ring) \xrightarrow{\simeq} S^*(\Map (S^1_{\bullet}, X);\ring),\]
where $QB^{\cyc}_\bullet S^*(X;\ring)$ is a resolution of the cyclic bar construction.
\end{Prop}
\begin{Rm}There is a more general statement when working with chains rather than cochains. See Remark 1 in \cite{UngherettiCyclicBar}.\end{Rm}
\begin{proof}[Proof of Proposition \ref{prop:bar_equivalence}] To extend the proof of the Main Theorem in \cite{UngherettiCyclicBar} two things need to be added. Lemma 1 on \cite{UngherettiCyclicBar} should be checked for the morphisms $r_n\in \dihedral^{\op}([n],[n])$, which is elementary. More importantly, one needs a contractible operad $\tilde{\SeqOp}$ with a natural action on cochains in such a way that it encodes the cup product in arity two, the involution $\inv{(-)}$ in arity one and insertion of the unit in arity zero. This ensures that the equation ($\star$) lives entirely inside the operad $\tilde{\SeqOp}$. In particular, the action of $r_n$ on the cyclic bar construction is a composition of a permutation of arguments (the operad is symmetric) and termwise application of the involution. The existence of an operad with such an action is established in Proposition \ref{prop:operad_exists} below.\end{proof}

\begin{Lem}\label{lem:dga_exists} There exists a quasi free unital differential graded algebra $(R,\delta)$ over $\ring$ and a natural differential graded $R\mhyphen$module structure on $S_*(X)$. The algebra $R$ contains a distinguished element $r$ of degree zero that acts on chains as $r \sigma = \inv{\sigma}$ where $\inv{\sigma}(t_0,\ldots,t_n)=(-1)^{n(n+1)/2}\sigma(t_n,\ldots, t_0)$.\end{Lem}
\begin{proof} We define $R=\bigcup_l R^l$ where we construct the $R^l$ inductively, starting with $R^0$ freely generated by an element $r$ of degree zero. From $R^l$ we obtain $R^{l+1}$ by adding a generator $h(a)$ for every word $a\in R^l$, where $h(a)$ is one degree higher than $a$ is. For example, $r^2 h(r^{12}h(r)h(rh(r^5)))$ is an element of degree four in $R^3$. The differential is defined on generators as $\delta h(r^p)=r^p - 1$ and $\delta h(a)=a-h(\delta a)$ if the degree of $a$ is not zero. The map $a\mapsto h(a)$ defines a contracting homotopy.

It now remains to show that we can define the natural module structure on $S_*(X)$. This is done by the method of acyclic models and induction on $l$, the degree of the operation and the degree of the chain on which the operations act. Note that if $a$ is any natural operation, it is determined by its action on universal simplices $\kappa_n\in S_n(\ford^n)$ for all $n$ as $a(\sigma)=a(\sigma_* \kappa_n)=\sigma_*(\kappa_n)$ for $\sigma\in S_n(X)$. Also, $a(d\sigma)=\sigma_* a(d\kappa_n)=\sigma_*\Sigma_{i=0}^{n} \delta^i_*a(\kappa_{n-1})$ where the maps $\delta^i\colon \ford^{n-1}\rightarrow \ford^n$ are the face inclusions. 

Because $R^0$ is freely generated by $r$, its action on chains is determined by the formula $r(\sigma)=\inv{\sigma}$. Finding the action of the $h(r^p)\in R^1$ reduces to fixing the action of a single $h(r)$ because $r^2(\sigma)=\sigma$. This operation can either be found using acyclic methods or can be found as the prism operator in the proof of homotopy commutativity of the cup product (see \cite[p.211]{HatcherAT}). We proceed with the inductive step. 

Assume we have specified the action of all generators in $R^{l-1}$ and all new generators (elements of the form $h(a)$ for $a\in R^{l-1}$) in $R^l$ of degree $m$. Let $a$ be a word of degree $m$ in $R^{l-1}$ that is not in $R^{l-2}$. We need to show the existence of a natural operation associated to $h(a)$ that satisfies $(\delta h(a))(\sigma)= a(\sigma) - h(\delta a)(\sigma)$. Here the first $\delta$ should be read as the differential as an operation: $(\delta h(a))(\sigma)=d h(a)(\sigma)- (-1)^{m+1} h(a)(d\sigma)$. Using this, we see that it suffices to specify $h(a)(\kappa_n)\in S_{n+m+1}(\ford^n)$ for all $n$, such that
\begin{equation}\label{eqn:boundary_condition}
d h(a)(\kappa_n)=a(\kappa_n) - h(\delta a)(\kappa_n)+ (-1)^{m+1} \Sigma_{i=0}^n (-1)^i \delta_*^i h(a)(\kappa_{n-1}).
\end{equation}
We can define such $h(a)(\kappa_n)$ by induction on $n$: If we assume to have found such $h(a)(\kappa_N)$ for $N < n$, all the terms of the right hand side are elements of $S_{n+m}(\ford^n)$ that have been found. A small calculation shows that the right hand side is a cycle and as $\ford^n$ is contractible, we see that $h(a)(\kappa_n)$ exists. For the base of this part of the induction one needs to find a $h(a)(\kappa_0)\in S_{m+1}(\ford^n)$ whose boundary is $a(\kappa_0)-h(\delta a)(\kappa_0)$, which is again possible because $\ford^n$ is contractible.
\end{proof}

\begin{Prop}\label{prop:operad_exists} There exists a symmetric, reduced, differential graded operad $\tilde{\SeqOp}$ with a natural action on cochains. The operad contains distinguished elements in arity two, one and zero representing the cup product, the involution and insertion of the unit respectively.\end{Prop}
\begin{proof}
The operad $\tilde{\SeqOp}$ is defined as the pushout of $\SeqOp \leftarrow \ring 1 \rightarrow R$. Here $\ring 1$ is the initial reduced operad and $R$ is the operad associated to the differential graded algebra of Lemma \ref{lem:dga_exists}. As $R$ is quasi free and contractible, $\ring 1 \hookrightarrow R$ is an acyclic cofibration in the Berger--Moerdijk model structure. Therefore $\SeqOp \xrightarrow{\simeq} \tilde{\SeqOp}$ is a weak equivalence. An algebra structure on a chain complex for such a coproduct operad is a pair of algebra structures that agree in arity zero. As we know that there are natural action of both $\SeqOp$ and $R$ on singular cochains and that they both insert the unit as the arity zero operation, we see that singular cochains carry a natural algebra structure over $\tilde{\SeqOp}$.

The distinguished elements in arity two and arity zero are provided by the image of $\SeqOp \rightarrow \tilde{\SeqOp}$. The involution is provided by the action of the generator $r\in R$.
\end{proof}

\section{Comparing homotopy orbits}\label{sec:compare_orbits}
An often used fact is that $X_{hG}\simeq (X_{hN})_{hG/N}$ for $X$ a $G$-space and $N\triangleleft G$. This is a consequence of the fact that any model for $EG$ is also a model for $EN$. We use a variation of this fact to see that $X_{h\OMon}\cong (X_{h\TMon})_{hC_2}$ in a way that is compatible with the algebraic statement of Proposition \ref{prop:alg_orbits}. 

\begin{Def}The two sided bar construction model for $E\TMon$ comes with a simplicial right $\OMon$-action that extends the $\TMon$-action: On $B_n(*,\TMon,\TMon)$ the action of $\OMon = \TMon\rtimes C_2$ is defined as $(z_1,\ldots,z_n)z_0 \cdot (z,\alpha)=(z_1^\alpha,\ldots, z_n^\alpha)z_0^\alpha z^\alpha$.
\end{Def}

\begin{Prop}\label{cor:homotopyorbits}Let $X$ be a left $\OMon\mhyphen$space. Then there is an equivalence $X_{h\OMon}\simeq ( X_{h\TMon})_{hC_2}$.\end{Prop}
\begin{proof}Let $EC_2$ be any contractible space with a free $C_2$-action. Then $C_2$ acts diagonally on $EC_2\times E\TMon$, whereas $\TMon$ acts only on $E\TMon$. Together this gives a free $\OMon$-action and thus a model for $E\OMon$. Using the fact that $(-)_\OMon=((-)_\TMon)_{C_2}$, we now see that $X_{h\OMon}=(EC_2\times E\TMon)\times_\OMon X = EC_2\times_{C_2}(E\TMon \times_\TMon X)=(X_{h\TMon})_{hC_2}$.\end{proof}

\begin{Prop}\label{prop:compare_orbits}Let $X$ be a left $\OMon\mhyphen$space. Then there are equivalences $S_*(X_{h\TMon})\simeq (S_*(X))_{h\TMon}$ and $S_*(X_{h\OMon})\simeq (S_*(X))_{h\OMon}$.\end{Prop}
\begin{proof}On both sides, the homotopy $\TMon\mhyphen$orbits can be described by a two sided bar construction. Combining the Eilenberg--Zilber equivalence with the map $\TMon=H_*(\TMon)\xhookrightarrow{\simeq} S_*(\TMon)$, we obtain an equivalence of simplicial chain complexes \[B_\bullet(\ring,\TMon,S_*X)\xrightarrow{\simeq}B_\bullet(\ring,S_*\TMon,S_*\TMon)\xrightarrow{\simeq}S_*(B_\bullet(*,\TMon,X)).\] Passing to the total complexes we obtain a quasi isomorphism
\[(S_*(X))_{h\TMon}\xrightarrow{\simeq}S_*(X_{h\TMon}).\]
Although this map is not $C_2$-equivariant on the nose, it is equivariant up to coherent homotopy, which is enough in order to compare homotopy orbits. More concretely, there exists a $C_2$-equivariant map \[B_*(C_2,C_2, B_*(\ring, \TMon, S_*(X)))\rightarrow B_*(\ring, S_*(\TMon), S_*(X)),\] and it is clear that this map induces an equivalence on homotopy orbits as claimed. The existence of the homotopy coherent map can be shown using acyclic methods.

The map $\TMon \hookrightarrow S_*(\TMon)$ does not commute with the $C_2$-action as $-[\TMon]$ and $[-\TMon]$ do not coincide in $S_*(X)$. Here $[-\TMon]$ is the fundamental cycle with the opposite orientation, given by $[-\TMon](t_0,t_1)=e^{- 2\pi i t_0}=e^{2\pi i t_1}$ as opposed to $[\TMon](t_0,t_1)=e^{2 \pi i t_0}$. The two are however homologous cycles, with the difference given as the boundary of the two-chain $P(t_0,t_1,t_2)=e^{2\pi i t_1}$ in the normalized complex. This gives the zero'th level of a $C_2$-equivariant map $B_*(C_2,C_2,\TMon)\xrightarrow{\simeq} S_*(\TMon),$ which exists by an acyclic methods argument. Taking tensor powers we obtain a sequence of maps
\[B_*(C_2,C_2,\TMon)^{\otimes n}\otimes S_*(X)\rightarrow (S_*(\TMon))^{\otimes n} \otimes S_*(X).\]
Using the Eilenberg--Zilber equivalence and the multiplication map $C_2^n\rightarrow C_2$ we get
\[B_*(C_2,C_2,\TMon)^{\otimes n}\xrightarrow{\simeq} B_*(C_2^n,C_2^n,\TMon^{\otimes n})\xrightarrow{\simeq}B_*(C_2, C_2, \TMon^{\otimes n}).\]
Combining these maps for all $n$ we get a map
\[B_*(C_2,C_2, B_*(\ring, \TMon, S_*(X)))\rightarrow B_*(\ring, S_*(\TMon),S_*(X)).\]
By taking $C_2$ homotopy orbits, we get the desired equivalence
\[(S_*(X))_{h\OMon}\simeq ((S_*(X))_{h\TMon})_{hC_2}\simeq S_*(X_{h\OMon}).\]\end{proof}

\section{Comparing totalisations}\label{sec:compare_tot}
Let $Y^\bullet$ be a codihedral space, for example $\Map(S^1_\bullet, X)$ for a space $X$. Associated to $Y^\bullet$ are two chain complexes and a natural map between them $\psi\colon S_*\tot Y^\bullet \rightarrow \Tot_\Pi S_*Y^\bullet$. This map is defined as $\psi(\sigma)=\Pi_n (\alpha_n)_*(\sigma \times \kappa_n)$, where $\kappa_n\in S_n(\ford^n)$ is the fundamental simplex and $\alpha_n \colon (\tot Y^\bullet) \times \ford^n \rightarrow Y^n$ is the evaluation map coming from the definition of totalization as the end construction $\tot Y^\bullet=\Nat_{\ford}(\stdSim^\bullet,Y^\bullet)$. Although both sides are $\OMon\mhyphen$modules, $\psi$ is only almost an $\OMon\mhyphen$map. To fix this, we now introduce a slightly different model for the right hand side. Let $\tilde{S}_n(X)$ denote the \emph{oriented singular $n\mhyphen$chains}, defined by quotienting $S_n(X)$ by the relation $g\cdot \sigma \sim \textit{sgn}(g) \sigma$ where $\textit{sgn}(g)$ is the sign of a permutation $g\in \Sigma_{n+1}$ that acts by permuting the coordinates simplices. This defines a functor $\tilde{S}\colon \Top \rightarrow \Ch$ that is naturally quasi isomorphic to the usual singular chains. For more details on $\tilde{S}$, see \cite{Barr}.

\begin{Prop}\label{prop:compare_tot}If $\psi$ is a quasi isomorphism, then $S_*\tot Y^\bullet$ and $\Tot_\Pi S_*Y^\bullet$ are quasi isomorphic as $\OMon\mhyphen$modules and \[(S_*\tot Y^\bullet)_{h\OMon} \simeq (\Tot_\Pi S_*Y^\bullet)_{h\OMon}.\]\end{Prop}
\begin{proof}It suffices to show that the composition \[\tilde{\psi}\colon S_*\tot Y^\bullet \xrightarrow{\psi} \Tot_\Pi S_* Y^n \rightarrow \Tot_\Pi \tilde{S}_* Y^n\] is an $\OMon\mhyphen$map. Concretely this means checking that it commutes with the $R$ and $B$ operators on both sides. On $\tot Y^\bullet$, the $C_2$-action is given by $(Rf)(\underline{t})=(\rho_nf)(\underline{t}^{\op})$, meaning that $(\alpha_n)_*(R\sigma \times \kappa_n)= (\rho_n \alpha_n)(\sigma \times \kappa^{\op})$. Here $\kappa^{\op}(\underline{t})=\underline{t}^{\op}=(t_n,\ldots,t_0)$. On the other hand, as the $C_2\mhyphen$action on $\Tot_\Pi S_*Y^\bullet$ is $(-1)^{n(n+1)/2}\rho_n$ on level $n$, the $R$ operator on the right hand side gives $R (\alpha_n)_*(\sigma\times \kappa_n)=(-1)^{n(n+1)/2} (\rho_n \alpha_n)_*(\sigma \times \kappa_n)$. By inspecting the definition of the shuffle product, it can be seen that these two expressions are equal when passing to $\tilde{S}_*$.

The analogous check for the $B$ operator involves comparing $\mu_*([\TMon]\times \iota_{n+1})$ with $B\iota_n$, where $\iota_n\in S_*(\stdCyclic^n)$ is the fundamental simplex. Again one needs to pass to $\tilde{S}_*$ for the two to be equal. A claim related to $\psi$ commuting with the $B$ operator is on page 417 of \cite{Jones}.\end{proof}

\begin{Rm}\label{rm:convergence}The condition that $\psi$ is a quasi isomorphism is not always satisfied and is related to the convergence of a generalized Eilenberg--Moore spectral sequence \cite{AndersonEMSS,BousfieldHomologySS}. When $\ring$ is a field and $Y^\bullet=\Map(S^1_\bullet,X)$ for $X$ a simply connected space, the condition is claimed to hold in \cite{AndersonEMSS} and a more detailed discussion is found in \cite{PatrasThomas}. Given that Jones' isomorphism has been proven to hold in greater generality in \cite{AyalaFrancis}, it is quite possible that $\psi$ is a quasi isomorphism under weaker hypotheses. This would strengthen our Main Theorem.\end{Rm}

\section{Proof of the Main Theorem}\label{sec:proof_main_theorem}
\begin{MainTh}Let $\ring$ be a field and $X$ a simply connected space of finite type over $\ring$. Then there is an isomorphism 
\[ H^*(\loops X_{h\OMon}) \cong HD^-_*(S^*(X)).\]\end{MainTh}
\begin{proof}
By Proposition \ref{prop:bar_equivalence} there is an equivalence of dihedral chain complexes $B^{\cyc}_\bullet S^*(X)\simeq S^*(\Map(S^1_\bullet,X))$, which by Example \ref{ex:total_dihedral} and a double complex argument gives an equivalence of differential graded $\OMon$-modules $C_*(S^*(X))=\Tot_\oplus B^{\cyc}_\bullet S^*(X)\simeq \Tot_\oplus S^*(\Map(S^1_\bullet,X))$. Applying the linear dual of Proposition \ref{prop:compare_tot} to the codihedral space $Y^\bullet=\Map(S^1_\bullet,X)$ yields an equivalence of differential graded $\OMon$-modules
\[\Tot_\oplus S^*(\Map(S^1_\bullet,X))\simeq_{\OMon} S^*(\tot \Map(S^1_\bullet,X)).\]
Note that $\tot \Map(S^1_\bullet,X)\cong_{\OMon} \loops X$ by Example \ref{ex:realize_circle} and that the hypothesis of Proposition \ref{prop:compare_tot} is satisfied because of Remark \ref{rm:convergence}. In all, we now have that $C_*(S^*(X))\simeq_{\OMon} S^*(\loops X)$ and Proposition \ref{prop:orbit_equivalence} implies
\[DC^-_*(S^*(X))=(C_*(S^*(X)))^{h\OMon}\simeq (S^*(\loops X))^{h\OMon}.\]
After applying the linear dual of Proposition \ref{prop:compare_orbits} we finally see that the last term is equivalent to $(S^*(\loops X))^{h\OMon}\simeq S^*(\loops X_{h\OMon})$ and the theorem follows.
\end{proof}

\begin{Rm}The exact same methods can be used to show a $C_2$-version of the Jones isomorphism. \[HR^-_*(S^*(X))\cong H^*(\loops X_{hC_2})\] The corresponding cohomology theory is called \emph{negative reflexive homology} $HR^-_*$ and is defined as the homology of $(C_*(A))^{hC_2}$.\end{Rm}

\section{An involutive de Rham isomorphism}\label{sec:deRham}
The goal of this section is to prove that the de Rham cochain algebra $\Omega_{dR}^*(M)$ and the singular cochain algebra $S^*(M)$ on a compact smooth manifold $M$ are quasi isomorphic as involutive dga's. To do this, we will take a zig-zag witnessing the quasi isomorphism without involutions, give all the terms involutions and check that all the maps in the zig-zag preserve these involutions. In particular this involves upgrading the polynomial de Rham forms $A^*_{PL}(M)$ to an involutive dga. The following pair of theorems is the starting point of the proof.

\begin{Th}[{\cite[Theorem 10.9]{FHT_RHT}}] \label{th:deRham1}
Let $\ring$ be a field of characteristic 0. For $K$ a simplicial set, the natural morphisms of differential graded algebras over $\ring$, 
\[A_{PL}(K) \rightarrow (C_{PL}\otimes A_{PL}(K))\leftarrow C_{PL}(K)\] 
are quasi isomorphisms.
\end{Th}

\begin{Th}[{\cite[Theorem 11.4]{FHT_RHT}}] \label{th:deRham2}
For a smooth manifold $M$, the natural morphisms of differential graded algebras over $\ring=\R$,
\[\Omega^*_{dR}(M)\xrightarrow{\alpha_M}A_{dR}(\Sing^\infty_\bullet(M))\xleftarrow{\beta_M}A_{PL}(\Sing^\infty_\bullet(M))\xleftarrow{\gamma_M} A_{PL}(\Sing_\bullet(M))\]
are quasi isomorphisms.
\end{Th}

All of the terms except for the smooth forms $\Omega^*_{dR}(M)$ can be defined using the following construction.
\begin{Def}\label{def:cochain_functor} For every simplicial dga $A_\bullet$, we get an associated \emph{cochain functor} $A(-) = \Nat_{\ford^{\op}} (-,A)\colon \sSet^{\op} \rightarrow \dga$. This construction is covariant in $A_\bullet$, so we have a functor $\sSet^{\op}\times \sdga \rightarrow \dga$ \end{Def}
\begin{Ex}Singular cochains of a space or more generally simplicial cochains of a simplicial set $K_\bullet$ can be viewed as $C_{PL}(K_\bullet)$, where $C_{PL \bullet}$ is the simplicial dga that is $C^*(\ford [ n ] )$ in simplicial degree $n$.\end{Ex}
\begin{Ex}Another important example is the \emph{piecewise linear forms} $A_{PL \bullet}$. In simplicial degree $n$ this is the cdga \[A_{PL}[n]=\Lambda(t_0,\ldots,t_n,dt_0,\ldots,dt_n)/(\Sigma t_i -1).\] Here the generators $t_i$ are of degree 0 and $\Lambda$ denotes the free graded commutative algebra. Because $A_{PL \bullet}$ is graded commutative in each simplicial degree, the associated cochain functor lands in cdga's.\end{Ex}
\begin{Ex}Smooth forms on the geometric simplices also form a simplicial cdga $A_{dR \bullet}$. The map $\beta_M$ is induced by the inclusion $A_{PL \bullet}\hookrightarrow A_{dR \bullet}$.\end{Ex}
\begin{Rm}The map $\alpha_M$ is induced by the maps $\sigma^*\colon \Omega^*_{dR}(M)\rightarrow \Omega^*_{dR}(\ford^{n}))$ that pull back forms along smooth simplices $\sigma\in \Sing^\infty _n(M)$. The map $\gamma_M$ is induced by the inclusion of smooth singular simplices into continuous singular simplices $\Sing^\infty_\bullet(M)\hookrightarrow S_\bullet(M)$. The last two maps come from the simplicial maps $A_{PL \bullet}\rightarrow C_{PL \bullet} \otimes A_{PL \bullet} \leftarrow C_{PL \bullet}$.
\end{Rm}

\subsection*{The involutions}
From now on, $\Omega^*_{dR}(M)$ will carry as its involution the identity, see Example \ref{ex:cdga}. For all the other terms, we extend Definition \ref{def:cochain_functor} so that it lands in involutive dga's. For this, we need to change the input\textemdash instead of simplicial sets, we use reflexive sets $\rSet=\Category{Set}^{\reflexive^{\op}}$ and instead of simplicial dga's we use the category $\invsdga$ defined below.

\begin{Def}An \emph{involutive reflexive dga} is a reflexive object in $\dga$ such that the reflexive structure map $r_n$ is an involution of dga's in every simplicial level $n$. The category of such objects with morphisms of reflexive dga's is called $\invsdga\subset \dga^{\reflexive^{\op}}$. Note that this is \emph{not} the same as simplicial objects in involutive dga's.\end{Def}
\begin{Ex}\label{ex:invsdga_pl}There is an involutive reflexive structure on $A_{PL \bullet}$ using $r_n\colon t_i\mapsto t_{n-i}$.\end{Ex}
\begin{Ex}\label{ex:invsdga_tensor}Given $A,B\in\invsdga$, their levelwise tensor product is again an involutive reflexive dga.\end{Ex}
\begin{Prop}Simplicial cochains $C_{PL \bullet}$ carry the structure of an $\invsdga$.\end{Prop}
\begin{proof}The standard simplicial sets $\ford [n]$ carry an involution by reversing order: $\inv{\sigma}(i)=n-\sigma(p-i)$ for $\sigma\in\ford [n]_p$. If we view these simplices as coming from the Yoneda embedding, we see that $\inv{\sigma\circ\sigma'}=\inv{\sigma}\circ\inv{\sigma'}$ and we see that the construction is both reflexive (in $p$) and coreflexive (in $n$). We now define the reflexive structure map on the simplicial dga $C_{PL \bullet}$ to be $\inv{\gamma}(\sigma)=(-1)^{p(p+1)/2}\gamma(\inv{\sigma})$, where $\gamma$ is a $p$ cochain in simplicial degree $n$. Unitality, $\inv{\inv{\gamma}}=\gamma$ and graded linearity are clear so we show the remaining properties.
\begin{description}
\item[Anti-simplicial] The simplicial structure of $C_{PL}$ is the cosimplicial direction of $\ford [n]_p$. For example, $\delta^i\colon \ford [n-1]\rightarrow\ford [n]$ pulls back to $d_i\colon C^*(\ford [n])\rightarrow C^*(\ford [n-1])$. It is then an elementary check to see that $\inv{d_i\gamma}(\sigma)=d_{n-i}\inv{\gamma}(\sigma)$ and similarly for the degeneracies.
\item[Differential] First we observe that on the level of simplicial chains we have $\inv{d\sigma}=(-1)^p d\inv{\sigma}$. Then it follows easily that $d\inv{\gamma}=\inv{d\gamma}$.
\item[Multiplication] Because the involution is anti-simplicial, we have the following identities by induction: For a $(p+q)$ simplex $\sigma$
\[d_{p+1}\ldots d_{p+q} \inv{\sigma} = \inv{d_o^q\sigma}\]
\[\inv{d_{q+1}\ldots d_{q+p}\sigma}=d_o^p\inv{\sigma}\]
The cup product of cochains $\gamma_1, \gamma_2$ of degrees $p$ and $q$ is defined as $\gamma_1\cup\gamma_2 (\sigma)=(-1)^{pq}\gamma_1(d_{p+1}\ldots d_{p+q}\sigma)\gamma_2(d_o^p\sigma)$. Now we have 
\begin{align*}
\inv{\gamma_1\cup\gamma_2}(\sigma)&=(-1)^{pq} (-1)^{(p+q)(p+q-1)/2}\gamma_1(d_{p+1}\ldots d_{p+q}\inv{\sigma)}\gamma_2(d_o^p\inv{\sigma})\\
&=(-1)^{pq} \gamma_1(\inv{d_o^q\sigma})\gamma_2(\inv{d_{q+1}\ldots d_{q+p}\sigma})\\
&=(-1)^{(p+q)(p+q-1)/2 + q(q-1)/2 + p(p-1)/2}\inv{\gamma_2}\cup\inv{\gamma_1}(\sigma)\\
&=(-1)^{pq}\inv{\gamma_2}\cup\inv{\gamma_1}(\sigma)
\end{align*}
\end{description}
\end{proof}

\begin{Prop}\label{prop:inv_cochain_functor}The construction of Definition \ref{def:cochain_functor} gives a functor $(\rSet)^{\op}\times\invsdga\rightarrow \invdga$ using the same dga $A(K)$ associated to the underlying simplicial dga of $A$ and underlying simplicial set $K$. This is defined to carry involution $\inv{\Phi}=(\sigma\mapsto I \Phi (R \sigma))$ for $\Phi\in A(K)$ with $I$ and $R$ the reflexive structure maps $r_n$ of $A$ and $K$ respectively.\end{Prop}
\begin{proof}
The fact that the construction is functorial follows immediately from the definition of the morphisms in $\rSet$ and $\invsdga$. What needs to be checked is that the map described really is an involution on $A(K)$. 
\begin{description}
\item[Target] The $\inv{\Phi}$ is an element of $A(K)=\Nat_\ford^{\op}(K,A)$, i.e., it is simplicial: \[\inv{\Phi}(d_i\sigma)=I\Phi (Rd_i\sigma)=I\Phi(d_{n-i}R\sigma)=Id_{n-i}\Phi(R\sigma)=d_i I\Phi(R\sigma)=d_i \inv{\Phi},\] and similarly for the degeneracy maps.
\item[Differential] The differential of a natural transformation $\Phi\in A(K)$ was defined using the target $A$. We see that \[d\inv{\Phi}(\sigma)=d_A I\Phi(R\sigma)=I(d_A\Phi)(R\sigma)=\inv{d\phi}(\sigma).\]
\item[Unitality] Is $\inv{1}=1\in A(K)$? The unit of $A(K)$ is defined to send any $\sigma\in K_n$ to the unit in simplicial degree $n$. So we see that \[\inv{1}(\sigma)=I 1(R\sigma)=1(R\sigma)=1(\sigma).\]
\item[Involution] The fact that $\inv{\inv{\Phi}}=\Phi$ follows from the properties $R\circ R=id_K$ and $I\circ I=id_A$.
\item[Multiplication] Let $\Phi,\Psi\in A(K)$ be of degree $p$ and $q$ respectively. We check that
\begin{align*}
\inv{\Phi \Psi}(\sigma)&=I(\Phi \Psi)(R\sigma)=I(\Phi(R\sigma)\Psi(R\sigma))\\
&=(-1)^{pq}(I\Psi(R\sigma))(I\Phi(R\sigma)) =(-1)^{pq}\inv{\Psi}(\sigma)\inv{\Phi}(\sigma)\\
&=(-1)^{pq}\inv{\Psi}\, \inv{\Phi}(\sigma).
\end{align*}
\end{description}
\end{proof}

\begin{Prop}\label{prop:sing_inv}The dga of singular cochains $S^*(X)=C_{PL}(\Sing_\bullet X)$ carries an involution given by $\inv{\gamma}(\sigma)=(-1)^{p(p+1)/2}\gamma(\inv{\sigma})$.\end{Prop}
\begin{proof}Combining the $\rSet\mhyphen$structure of $\Sing_\bullet(X)$ with the $\invsdga$ structure of $C_{PL}$ supplies the singular cochain dga $S^*(X)$ with an involution. It is useful to see what it does concretely. First we will describe the isomorphism $S^*(X)\cong C_{PL}(\Sing_\bullet(X))$ precisely. Let $\lambda \in C_{PL}(\Sing_\bullet(X))$, say of degree $p$: A map that assigns to every non-degenerate $n$ simplex in $\Sing_n(X)$ a $p$ cochain in $(C_{PL})_n=C^p(\ford [n])$. It corresponds to the singular cochain in $S^p(X)$ that sends $\sigma \mapsto \lambda_\sigma (c_p)$ where $c_p$ is the fundamental simplex of $\ford[p]$, that is $id_p\in \ford([p],[p])$. The other way around, given a $\gamma\in S^p(X)$, the corresponding element of $C_{PL}(\Sing_\bullet(X))$ sends $\sigma \mapsto C^p(\sigma_*)(\gamma)=(\tau\mapsto\gamma(\sigma_*\tau))$, where $\sigma_*\colon\ford[n]\rightarrow \Sing_\bullet(X)$ using the Yoneda lemma. We now chase the involution.
\begin{align*}
\gamma &\mapsto (\sigma \mapsto C^p(\sigma_*)(\gamma))\\
&\mapsto (\sigma\mapsto \inv{\sigma} \mapsto C^p(\sigma_*)(\gamma) \mapsto (\ford[n]_p \ni \tau \mapsto (-1)^{p(p+1)/2}\gamma(\inv{\sigma}_* \inv{\tau})))\\
&\mapsto (\sigma\mapsto (-1)^{p(p+1)/2} \gamma(\inv{\sigma}_* \inv{e_p})=(-1)^{p(p-1)/2}\gamma(\inv{\sigma}))
\end{align*}
In the last line we use the fact that the fundamental simplices are fixed by the involution $\inv{e_p}=e_p$ and that $\inv{\sigma}_* e_p = \inv{\sigma}$. So we see that all the involution does is add a sign and evaluate on the flipped simplex: $\inv{\gamma}(\sigma)=(-1)^{p(p+1)/2}\gamma(\inv{\sigma})$. The exact same holds for smooth simplices.
\end{proof}

\begin{Th}The maps in the zig-zags of Theorems \ref{th:deRham1} and \ref{th:deRham2} are maps of involutive dga's.\end{Th}
\begin{proof}
The only map not given by bifunctoriality of the `cochain functor' construction is $\alpha_M$. It is given by sending a form $\omega\mapsto \{\sigma^*\omega\}_{\sigma\in \Sing^\infty_\bullet(M)}$. As $\Omega^*_{dR}(M)$ is graded commutative, the identity map will act as the involution. So to check that $\alpha_M$ respects the involution it is equivalent to check that the image is pointwise fixed by the involution. The involution in the target $A_{dR}(\Sing^\infty_\bullet(M))$ is given by the combination of $I_{A_{dR}}$ and $R_{\Sing^\infty_\bullet}(M)$. The first pulls forms on $\ford^n$ back along $\phi_n$ (the map that flips coordinates), the second flips coordinates of smooth singular simplices $R\sigma=\sigma\circ\phi = \inv{\sigma}$. \[\inv{\{\sigma^*\omega\}}=\{I(R\sigma)^*\omega\}=\{\phi^*(\sigma \circ \phi)^*\omega\}=\{\phi^*\phi^*\sigma^*\omega\}=\{\sigma^*\omega\}\]

The map $\beta_M$ is induced by the inclusion of piecewise linear forms into de Rham forms on $\ford^n$, $A_{PL}\hookrightarrow A_{dR}$, which is clearly a morphism in $\invsdga$. So $\beta_M$ respects the involution by functoriality. The same holds for $\gamma_M$, which is induced by the morphism $\Sing^\infty_\bullet(M) \hookrightarrow \Sing_\bullet(M)$ in $\rSet$.

Finally, it is an elementary check that the tensor product of simplicial dga's can be promoted to a tensor product in $\invdga$ (see Example \ref{ex:invsdga_tensor}) and that $A_{PL}\rightarrow C_{PL}\otimes A_{PL}\leftarrow C_{PL}$ are morphisms in $\invdga$. Hence the last three maps in the zig-zag respect the involution.
\end{proof}
\begin{Cor}De Rham cochains $\Omega^*_{dR}(M)$ with trivial involution and normalized singular cochains $S^*(M)$ with the involution from Proposition \ref{prop:sing_inv} are quasi isomorphic as involutive dga's.\end{Cor}

\begin{Cor}Let $M$ be a simply connected manifold of finite type over $\R$. Then using the trivial involution on $\Omega^*(M;\R)$ the following isomorphisms hold.
\begin{align*}H^*(\loops M;\R) &\cong HH_*(\Omega^*(M;\R))\\
H^*(\loops M_{h\TMon};\R) &\cong HC_*^-(\Omega^*(M;\R))\\
H^*(\loops M_{h\OMon};\R) &\cong HD_*^-(\Omega^*(M;\R))
\end{align*}
\end{Cor}
Note that the first two isomorphisms already follow from Theorem A in \cite{Jones}, combined with the de Rham Theorem. The proofs do imply however, that these two isomorphisms are $C_2$-equivariant.
\begin{Prop}\label{prop:commutative_cochain_functor}Let $K_\bullet$ be a reflexive set (e.g., $\Sing_\bullet X$) and $A$ an involutive reflexive dga (e.g., $A_{PL}$). If $A$ is graded commutative in every simplicial degree and the involution $\inv{(-)}$ of Proposition \ref{prop:inv_cochain_functor} is chain homotopic to the identity, then $(A(K),\inv{(-)})\simeq (A(K), id)$ as involutive dga's.\end{Prop}
\begin{proof}As both $K_\bullet$ and $A$ are reflexive objects, we may form the end construction $\tilde{A}(K) \coloneqq \Nat_{\reflexive}(K_\bullet,A)\subset \Nat_{\ford}(K_\bullet,A)=A(K)$. On $\tilde{A}(K)$, the two involutions agree and the inclusion map $\iota\colon \tilde{A}(K)\hookrightarrow A(K)$ is a section of the chain map $\rho\colon \Psi \mapsto \frac{1}{2} (\Psi + \inv{\Psi})$. If $h$ is a chain homotopy between the involution and the identity, then $\frac{1}{2}h$ is a chain homotopy between $\iota\circ\rho$ and the identity. Hence $\iota$ is a quasi isomorphism and the result follows from the following zig-zag. 
\[(A(K),\inv{(-)})\hookleftarrow (\tilde{A}(K), id) \hookrightarrow (A(K), id)\]\end{proof}

\begin{Prop}\label{prop:sing_equivalence}Let $X$ be a topological space. Then \[(A_{PL}(\Sing_\bullet X),\inv{(-)})\simeq (A_{PL}(\Sing_\bullet X), id)\] as involutive dga's.\end{Prop}
\begin{proof}We need to show that the condition of Proposition \ref{prop:commutative_cochain_functor} holds for $K_\bullet = \Sing_\bullet X$ and $A=A_{PL}$. This can be seen by considering the following diagram, which commutes by Theorem \ref{th:deRham1}.

\[\xymatrix{
A_{PL}(\Sing_\bullet X) \ar[d]^{\inv{(-)}} \ar[r] & (C_{PL}\otimes A_{PL}(\Sing_\bullet X)) \ar[d]^{\inv{(-)}} & S^*(X) \ar[l] \ar[d]^{\inv{(-)}} \\
A_{PL}(\Sing_\bullet X) \ar[r] & (C_{PL}\otimes A_{PL}(\Sing_\bullet X)) & S^*(X) \ar[l] 
}\]

After taking homology, the vertical morphism on the right hand side is the identity and arrows are isomorphisms. This implies that on homology, the involution $\inv{(-)}$ on $A_{PL}(K)$ is the identity. Over a field, two chain maps are the same on homology if and only if they are chain homotopic and thus it follows that $\inv{(-)}$ is homotopic to the identity.
\end{proof}

The fact above allows one to take any cdga model $A$ for a space $X$ from rational homotopy theory and use the equivalence $(A,id)\simeq (S^*(X;\Q),\inv{(-)})$ to compute $H^*(\loops X_{h\OMon})$ with. In particular we have the following statement.

\begin{Cor}\label{cor:rational_formality}If $X$ is a rationally formal simply connected space of finite type, then $S^*(X)$ is formal as an involutive dga and thus
\[H^*(\loops X_{h\OMon};\Q) \cong HD_*^-(H^*(X;\Q)).\]
\end{Cor}

\section{An example calculation}\label{sec:example_calc}
In this section we show that the results from the last section allow one to do concrete calculations. To demonstrate this we calculate Borel equivariant cohomology of $\loops S^2$ over the rationals. As $S^2$ is rationally formal, it is involutively formal over the rationals by Corollary \ref{cor:rational_formality} and to calculate negative dihedral homology we may use the algebra $H^*(S^2;\Q)=\Q[\alpha]/\alpha^2$ where $\degree{\alpha}=2$. 

The normalized Hochschild complex is generated by classes $\alpha_n=1\otimes \alpha^{\otimes n}$ and $\beta_n=\alpha \otimes \alpha^{\otimes n}$ for all $n\geq 0$, which have total degrees $-n$ and $-(n+1)$ respectively. As the internal differential is $0$, the total differential is given by $(-1)^{\textit{int}}\Sigma (-1)^id_i$. By using that $\alpha^2=0$, it is easy to calculate the differential. The only classes on which the differential is not zero are $d\alpha_n=2\beta_{n-1}$ for even $n$.

From this one can see that the Hochschild homology and therefore also $H^*(\loops S^2;\Q)$ is generated as a graded vector space by the classes
\[\alpha_0, \alpha_1, \alpha_3, \alpha_5, \ldots \quad \textit{and} \quad \beta_0, \beta_2, \beta_4,\beta_6, \ldots\] 
That is, there is exactly one class in every negative degree.

By using that the involution is the identity, it can be seen that the $R$ operation is $R=(-1)^{n(n+1)/2}id$ in simplicial degree $n$. A quick calculation shows that the only classes on which $B$ acts non-trivially are the classes $B\beta_n=(n+1)\alpha_{n+1}$ for $n$ even.

In order to now calculate the homotopy orbits for the $C_2$-action we can use the following well known fact.
\begin{Prop}\label{prop:2invertible}Let $\ring$ be any ring with $\frac{1}{2}\in\ring$. And let $W$ be a $C_2$-space. Then $H^*(W_{hC_2};\ring)\cong H^*(W;\ring)^{C_2}$.\end{Prop}
\begin{proof}When 2 is invertible, any element $m$ of a $\ring C_2$-module $M$ can be projected to the invariant element $\frac{1}{2}(m+gm)$, where $g$ is the generator of $C_2$. Using this, one can check that the functor of $C_2$-invariants is exact. This implies that the group cohomology over $\ring$ is $H^*(C_2;M)=0$ for $*>0$ and $H^0(C_2;M)=M^{C_2}$.
The map $W\twoheadrightarrow pt$ induces a fibration
\[EC_2\times_{C_2}W\twoheadrightarrow EC_2\times_{C_2}pt\simeq BC_2.\]
The associated Leray--Serre spectral sequence is \[E_2^{p,q}=H^p(BC_2;H^q(W))\convergesto H^{p+q}(W_{hC_2})\] Reinterpreting the twisted coefficients on the $E_2$ page as group cohomology, we see that the spectral sequence collapses here and read off the conclusion
\[H^q(W_{hC_2})\cong H^0(C_2;H^q(W))=H^q(W)^{C_2}.\]
\end{proof}

\begin{Cor}There are isomorphisms
\[H^*((\loops S^2)_{hC_2};\Q)\cong H^*(\loops S^2;\Q)^{C_2}\cong HH_{-*}(\Q[\alpha]/\alpha^2)^{C_2}.\]
As a graded module, this is generated by the classes $\alpha_0,\alpha_3,\alpha_7,\alpha_{11},\ldots$ and $\beta_0, \beta_4, \beta_8,\ldots$. In particular, this is four periodic.\end{Cor}

In order to now calculate the negative cyclic and negative dihedral homology, we consider the negative cyclic chains as the totalization of a double complex. In general we ought to use the product totalization, but in our case this coincides with the sum totalization because of the coconnectivity of the normalized Hochschild complex. The double complex gives us to converging spectral sequences. In particular, we consider the spectral sequence
\[(E^1, d^1) = (HH_*( \Q [\alpha] / {\alpha}^2)[u], uB) \convergesto HC^-_*(\Q[\alpha]/{\alpha}^2).\]
On the $E^1$ page, there is exactly one generator in each bidegree above or on the diagonal in the third quadrant. By considering the fact that $B$ maps every surviving $\beta_n$ class to a multiple of $\alpha_{n+1}$, we see that none of the classes on the interior survive to $E^2$. On $E^2$ we are left with the classes $u^p\alpha_0$ and $\alpha_q$ for odd $q$. Because of their degrees it is possible that $d^p$ maps $\alpha_{2p-1}$ to a multiple of $u^p\alpha_0$. But, because nothing can kill the $\alpha_{2p-1}$ in the double complex, we see that in fact all the differentials must be zero and hence $E^2=E^\infty$ and we can read off the cyclic homology. And to compute the negative dihedral homology, we can again apply Proposition \ref{prop:2invertible}. Note that the generator of $C_2$ acts as $u^p\alpha_0\mapsto (-1)^pR(\alpha_0)=(-1)^p\alpha_0$.

\begin{Th}\label{Th:rational_sphere}As a graded vector space, $H^*((\loops S^2)_{h\TMon};\Q)$ is generated by the classes $\alpha_0, \alpha_1, \alpha_3, \alpha_5,\ldots$ and $u\alpha_0,u^2\alpha_0,u^3\alpha_0,\ldots$. In other words, it is one dimensional in every degree. The cohomology $H^*((\loops S^2)_{h\OMon};\Q)$ is generated by the classes $\alpha_0,\alpha_3,\alpha_7,\alpha_{11},\alpha_{15},\ldots$ and $u^2\alpha_0, u^4\alpha_0, u^6\alpha_0,\ldots$ as a graded vector space.
\end{Th}

\begin{Rm}Although $S^*(S^2;\F_2)$ is formal as a dga over $\F_2$, it is not clear to the author whether $S^2$ is involutively formal over $\ring=\F_2$. Assuming it is, we can again use negative dihedral homology of $\F_2[\alpha]/\alpha^2$ to compute $H^*(\loops S^2 _{h\OMon};\F_2)$. The complex that computes the negative dihedral homology is generated by $v^p u^q \alpha_n$ and $v^p u^q \beta_n$ for all $n,p,q\geq 0$ and the only non-trivial differential is $v^p u^q \beta_n \mapsto v^p u^{q+1} \alpha_{n+1}$. This means that the cohomology is generated by the classes $v^p u^q \beta_n$ for odd $n$, $v^p u^q \alpha_n$ for even $n$ and $u^q \alpha_n$ for all $n$. According to a computer calculation for low degrees and \cite{OEISSequence}, this results in Betti numbers that are $\lfloor (*+2)^2/4 \rfloor$, which is a monotonic sequence.\end{Rm}
{
\printbibliography
}
\end{document}